\documentclass[11pt]{amsart}
\usepackage[utf8]{inputenc}
\usepackage[T1]{fontenc}
\usepackage{graphicx}
\usepackage{amsmath}
\usepackage{amsfonts}
\usepackage{graphicx}

\usepackage{mhequ}
\usepackage{comment}
\usepackage{dsfont}

\usepackage{xcolor}

\usepackage{url}

\usepackage{hyperref}
\usepackage[margin=3.5cm]{geometry}

\newtheorem{theo}{Theorem}[section]

\newtheorem{rem}[theo]{Remark}

\newtheorem{propo}[theo]{Proposition}
\newtheorem{lemme}[theo]{Lemma}
\newtheorem{defi}[theo]{Definition}
\newtheorem{ex}[theo]{Example}

\newtheorem{hyp}[theo]{Assumption}
\newcommand{\E}{\mathbb{E}}
\newcommand{\R}{\mathbb{R}}
\newcommand{\Z}{\mathbb{Z}}
\newcommand{\PP}{\mathbb{P}}

\newcommand{\N}{\mathbb{N}}
\newcommand{\HH}{\mathcal{H}}

\newcommand{\s}{\textbf{s}}
\newcommand{\rr}{ m}

\newcommand{\U}{Y}

\colorlet{darkred}{red!70!black}

\def\Tr{\mathop{\mathrm{Tr}}}

\title[Weak error for trajectories of SPDE{\tiny s}]{Weak error estimates for trajectories of SPDE{\tiny s} under Spectral Galerkin discretization}

\author[C.-E. Br\'ehier]{Charles-Edouard  Br\'ehier}
\author[M. Hairer]{Martin Hairer}
\author[A.~M. Stuart]{Andrew~M. Stuart}
\address[C-E Br\'ehier]{Univ Lyon, Universit\'e Claude Bernard Lyon 1, CNRS UMR 5208, Institut Camille Jordan, 43 blvd. du 11 novembre 1918, F-69622 Villeurbanne cedex, France \newline
Mail: brehier@math.univ-lyon1.fr}

\address[M. Hairer, A.~M. Stuart]{\newline Mathematics Institute, University of Warwick, Coventry CV4 7AL,
\newline UK \newline
Mail: M.Hairer@warwick.ac.uk, A.M.Stuart@warwick.ac.uk
}

\begin{document}

\begin{abstract}
We consider stochastic semi-linear evolution equations which are 
driven by additive, spatially correlated, Wiener noise, and in particular consider problems of heat equation (analytic semigroup) and damped-driven wave equations (bounded semigroup) type. We discretize these equations by means of a spectral Galerkin projection, and we study the approximation of the probability distribution of the trajectories: test functions are regular, but depend on the values of the process on the
interval $[0,T]$.

We introduce a new approach in the context of quantative weak error analysis for discretization of SPDEs. The weak error is formulated using a deterministic function (It\^o map) of the stochastic convolution found when the
nonlinear term is dropped. The regularity properties of the It\^o map
are exploited, and in particular second-order Taylor expansions employed,
to transfer the error from spectral approximation of the stochastic 
convolution into the weak error of interest. 

We prove that the weak rate of convergence is twice the strong rate of convergence in two situations. First, we assume that the covariance operator commutes with the generator of the semigroup: the first order term in the weak error expansion cancels out thanks to an independence property. Second, we remove the commuting assumption, and extend the previous result, thanks to the analysis of a new error term depending on a commutator.



\end{abstract}

\maketitle

\section{Introduction}

The numerical analysis of stochastic differential equations (SDEs), in both the
weak and strong senses, has been an active area of research over the
last three decades \cite{kloeden1992numerical,MilsteinTretyakov}.  
The analysis of numerical methods for stochastic partial differential 
equations (SPDEs) has attracted a lot of attention and in recent years
a number of texts have appeared in this field; see for instance
the recent monographs~\cite{JentzenKloeden:11}, \cite{Kruse:14} 
and~\cite{LordPowellShardlow:14}. The aim of this article is to give a simple argument
allowing to relate the weak order to the strong order of convergence on the space of trajectories
for a class of spatial approximations to SPDEs. 

We focus on the following class of semilinear SPDEs, 
written using the stochastic evolution equations framework in Hilbert spaces, from \cite{DPZ}:
\begin{equation}\label{eq:SPDE_intro}
dX(t)=\mathcal{A}X(t)dt+\mathcal{F}\left(X(t)\right)dt+d\mathcal{W}^{\mathcal{Q}}(t), \; X(0)=x_0.
\end{equation}
The semi-linear equation~\eqref{eq:SPDE_intro} is driven by an additive Wiener process $\mathcal{W}^{\mathcal{Q}}$, where $\mathcal{Q}$ is a covariance operator.
The following parabolic, resp. hyperbolic, SPDEs can be written as~\eqref{eq:SPDE_intro}, with appropriate definitions of the coefficients $\mathcal{A}$, $\mathcal{F}$ and $\mathcal{Q}$ in terms of $A$, $F$ and~$Q$:
\begin{itemize}
\item {\em the semi-linear stochastic heat equation} (parabolic case), with $X=u$, 
\begin{equation}\label{eq:SPDE_para_intro}
du(t)=Au(t)dt+F\left(u(t)\right)dt+dW^Q(t), \; u(0)=u_0;
\end{equation}
\item {\em the damped-driven semi-linear wave equation} (hyperbolic case), with $X=(u,v)$
\begin{equation}\label{eq:SPDE_wave_intro}
\begin{cases}
du(t)=v(t)dt\\
dv(t)=-\gamma v(t)dt+Au(t)dt+F(u(t))dt+dW^{Q}(t).
\end{cases}
\end{equation}
\end{itemize}
These two equations will be the focus of our work.
Notation and assumptions on the coefficients are precised in Section~\ref{sect:Not_Ass} below. For simplicity, in this introductory section, we assume that $F:H\to H$ is of class $\mathcal{C}^2$.

The solution $X$ of~\eqref{eq:SPDE_intro} (well-posed under assumptions given below) is a continuous-time stochastic process taking values in a separable, infinite-dimensional Hilbert space, which we denote by $H$. As for deterministic PDE problems, two kinds of discretizations are required in order to build practical algorithms: a time-discretization, which in the stochastic context
is often a variant of the Euler-Maruyama method, and a space-discretization, which is based on finite differences, finite elements or spectral approximation. In this article, we only study the space-discretization error (no time-discretization), using a spectral Galerkin projection, {\it i.e.} by projecting the equation on vector spaces spanned by $N$ eigenvectors of the linear operator $A$. Precisely, $X$ is approximated by the solution $X_N$ of an equation of the form
\begin{equation}\label{eq:SPDE_projN_intro}
dX_N(t)=\mathcal{A}_NX_N(t)dt+
\mathcal{F}_N\left(X_N(t)\right)dt+d\mathcal{W}^{\mathcal{Q}_N}(t),
\end{equation}
where the coefficients $\mathcal{A}_N$, $\mathcal{F}_N$, $\mathcal{Q}_N$ and the initial condition $X_N(0)$ are defined using the orthogonal projection $P_N\in\mathcal{L}(H)$ onto the $N$-dimensional vector space spanned by $e_1,\ldots,e_N$, where $Ae_n=-\lambda_ne_n$, for all $n\in\N$, with $\lambda_{n+1}\ge \lambda_n\ge \lambda_1>0$.

\bigskip

When looking at rates of convergence for the discretization of SPDEs, the metric one uses to compare random 
variables plays an important role. Let $\mathcal{Z}$, resp. $(\mathcal{Z}_n)_{n\in \{1,2,\ldots\}}$, 
denote a random variable, respectively a sequence of random variables, defined on a probability space $(\Omega,\mathcal{F}_\Omega,\mathbb{P})$, with values in a Polish space $E$ (separable and complete metric space, with distance denoted by $d_E$).
Strong approximation is a pathwise concept, typically defined through 
convergence in the mean-square sense of $\mathcal{Z}_n$ to $\mathcal{Z}$, 
{\it i.e.} the convergence of the strong error $e_{n}^{\rm strong}=\bigl(\mathbb{E} d_E(\mathcal{Z},\mathcal{Z}_n)^2\bigr)^{1/2}$, or in an almost sure
sense; see \cite{kloeden1992numerical} for details.
Weak approximation corresponds to convergence in distribution of $\mathcal{Z}_n$ to $\mathcal{Z}$, 
which is often encoded in a weak error of the type $e_{n}^{\rm weak}=\sup_{\varphi\in\mathcal{C}}|\mathbb{E}\varphi(\mathcal{Z})-\mathbb{E}\varphi(\mathcal{Z}_n)|$, for some class $\mathcal{C}$ of sufficiently regular test functions $\varphi:E\to \R$.
If functions in $\mathcal{C}$ are uniformly Lipschitz continuous, it follows that $e_{n}^{\rm weak}=\text{O}(e_{n}^{\rm strong})$. The problem we address in our situation is to show (and quantify) that $e_{n}^{\rm weak}=\text{o}(e_{n}^{\rm strong})$. In many situations, it is known that
the weak order of convergence is twice the strong order. We establish, for
spectral approximation of SPDEs, situations where the weak order
exceeds the strong order and where it is, in some cases, precisely
twice the strong order. 

As the references below will substantiate, the ``weak twice strong''
type of result has been proved when $E=H$ and $\mathcal{Z}=X(T)$ for SPDE~\eqref{eq:SPDE_intro}, with sufficiently regular test functions, {\it i.e.} looking at the processes at a given deterministic time $T\in(0,+\infty)$, and a variety
of approximation methods.
In this article, we focus on a more difficult problem: the weak approximation in the separable Banach space $E=\mathcal{C}\bigl([0,T],H\bigr)$, referred to as the space of trajectories. In other words, we consider $\mathcal{Z}=\bigl(X(t)\bigr)_{t\in[0,T]}$. The class $\mathcal{C}$ of test functions is taken as a bounded subset of $\mathcal{C}_b^2(E,\R)$, the Banach space of functions $\varphi:E\to\R$ which are bounded, and admit first and second-order bounded and continuous derivatives.

We now review the literature on weak approximation of SPDEs driven by Wiener 
noise. Our aim is not to give an exhaustive list of references, but to focus on three approaches which 
have been studied in the case of semilinear SPDEs, with low spatial regularity of the noise perturbation.
In a first approach, one relies on a representation formula for the weak error, using the solution of a deterministic evolution PDE (Kolmogorov equation), depending on an $H$-valued variable. This approach is a generalization of the well-known method used to study the weak error of time-discretization schemes of SDEs, {\it i.e.} finite dimensional diffusion processes; see for instance~\cite{GrahamTalay},~\cite{MilsteinTretyakov} and references therein. For linear equations perturbed with additive-noise, see \cite{DebusschePrintems:09}, \cite{GeissertKovacsLarsson:09}, \cite{KovacsLarssonLindgren:12}, \cite{KovacsLarssonLindgren:13}.
These works however use a specific change of variables, and this trick does not seem to work for semilinear equations. 
A related approach using Malliavin calculus techniques to prove (when the noise has low spatial regularity) that the weak order is twice the strong order is available, see \cite{Debussche:11} for the original arguments, and \cite{AnderssonLarsson:13}, \cite{Brehier:14}, \cite{WangGan:13} for some extensions, and \cite{Hausenblas:03}, \cite{Wang:16} for related works. In a second approach, considered in \cite{ConusJentzenKurniawan:14}, \cite{JacobeJentzenWelti:15}, \cite{JentzenKurniawan:15}, one expands the weak error using a mild It\^o formula. This technique allows to improve the results obtained with the first approach, for SPDEs driven by multiplicative noise.
Finally, in a third approach, described in \cite{AnderssonKovacsLarsson:14}
and \cite{AnderssonKruseLarsson:13} for example, 
one estimates the weak error using a duality argument.

These approaches do not seem to apply to the weak approximation problem in the space of trajectories. 
Instead, we adopt an approach used in
\cite{mattingly2012diffusion,pillai2012optimal,PillaiStuartThiery:14} 
(see also the references therein) to prove diffusion limits for MCMC
methods. The idea goes as follows: write $\mathcal{W}^{\mathcal{A},\mathcal{Q}}$ for the stochastic convolution, 
i.e.\ the solution of the SPDE~\eqref{eq:SPDE_intro} when $\mathcal{F}=0$ and build a map 
$\Theta:\mathcal{C}\bigl([0,T],H\bigr)\to \mathcal{C}\bigl([0,T],H\bigr)$ so that the corresponding solution 
$X$ to the full SPDE is given by $X=\Theta(\mathcal{W}^{\mathcal{A},\mathcal{Q}})$.
The key point is that one can then find an error term $\mathcal{R}_N$ 
(which encodes the error due to the approximation of the initial condition and of the semilinear coefficient)
so that the solution $X_N$ to the approximate equation can be written as
\begin{equ}
X_N=\Theta(\mathcal{W}^{\mathcal{A}_N,\mathcal{Q}_N}+\mathcal{R}_N)\;,
\end{equ}
where $\Theta$ is the exact same map as above.
To simplify the discussion, in the present section, we do not discuss the role of the error term $\mathcal{R}_N$: indeed, it usually converges faster to $0$ than the contribution due to the approximation of the stochastic convolution.
We note two differences with~\cite{PillaiStuartThiery:14}. First, we express the process of interest $X$ in terms of the stochastic convolution $\mathcal{W}^{\mathcal{A},\mathcal{Q}}$, not in terms of the Wiener process $\mathcal{W}^{\mathcal{Q}}$; this proof approach is related to the structure of the SPDEs we are interested in. Second, we derive rates of weak convergence, using $\mathcal{C}^2$-regularity of the It\^o map, instead of only utilizing perservation
of weak convergence under continuous mappings.

It is now easy to explain why the weak order of convergence is expected to be twice the strong order one, under the condition that $Q$ commutes with $A$. Indeed, $\Theta$ is of class $\mathcal{C}^2$, with bounded first and second order derivatives. In particular, $\Theta$ is Lipschitz-continuous and thus, neglecting the contribution of $\mathcal{R}_N$,
 the strong error is expected to be of order
\begin{equ}
e_N^{\rm strong} = \bigl(\E\sup_{t\in[0,T]}|X(t)-X_N(t)|^2\bigr)^{1/2} \approx \bigl(\E\sup_{t\in[0,T]}|\mathcal{W}^{\mathcal{A},\mathcal{Q}}(t)-\mathcal{W}^{\mathcal{A}_N,\mathcal{Q}_N}(t)|^2\bigr)^{1/2}\;.
\end{equ}
To control the weak error, one relies on a Taylor second-order expansion of the function $\varphi\circ\Theta$ (where $\varphi:\mathcal{C}([0,T],H)\to \R$ is a test function of class $\mathcal{C}^2$). The key argument is then the independence of the processes $\mathcal{W}^{\mathcal{A}_N,\mathcal{Q}_N}$ and $\mathcal{W}^{\mathcal{A},\mathcal{Q}}-\mathcal{W}^{\mathcal{A}_N,\mathcal{Q}_N}$,
which is a consequence of using the projection operator $P_N$ onto a space $H_N$ spanned  by eigenvectors of both 
$A$ and $Q$ (thanks to the assumption that $Q$ commutes with $A$). Then the expectation of the first-order term in the Taylor expansion vanishes, which proves that the weak error is of size $e_N^{\rm weak}\leq C\E\sup_{t\in[0,T]}|\mathcal{W}^{\mathcal{A},\mathcal{Q}}(t)-\mathcal{W}^{\mathcal{A}_N,\mathcal{Q}_N}(t)|^2 \approx (e_N^{\rm strong})^2$.

Our main result is then Theorem~\ref{theo:weak} (in the parabolic case \eqref{eq:SPDE_para_intro}): it provides the weak order of convergence for trajectories of SPDEs in a general setting (depending on regularity assumptions on $F$ and $Q$, see Section~\ref{sect:Not_Ass}). In the hyperbolic setting, this result is given by Theorem~\ref{theo:wave}.

In the case when $A$ and $Q$ do not commute, the key independence argument above breaks down and it is not clear {\it a priori} whether the weak order is still twice, or at least larger than, the strong order. In Section~\ref{sect:non_comm}, we give a control of the weak error in this non-commuting situation (Theorem~\ref{theo:weak_gen}), with an additional error term which depends on an auxiliary stochastic convolution where commutators appear. We are able to control this additional term in the case where the operator $Q$
is a multiplication operator, in Section~\ref{sect:example_non_comm}: we provide a non-trivial extension of Theorem~\ref{theo:weak} in Theorem~\ref{theo:weak_special} where there is no commutativity,
but the weak order is still twice the strong order. 


Note that the arguments described above do not directly apply if one considers weak approximation of trajectories associated with discretization in time, or discretization in space using finite elements. These situations will be investigated in future work.

\bigskip

The article is organized as follows. In Section~\ref{sect:Not_Ass}, we present the (regularity, growth) assumptions and notation to ensure well-posedness of the parabolic SPDE~\eqref{eq:SPDE_intro}. In Section~\ref{sect:Ito_parabolic}, we introduce one of the main tools for our convergence analysis, namely the It\^o map (Section~\ref{sect:Itomap}). The spectral Galerkin discretization method is introduced in Section~\ref{sect:Galerkin_parabolic}.
The strong and weak orders of convergence, in the case when $A$ and $Q$ commute, are provided in Theorem~\ref{theo:strong}, and our main result, Theorem~\ref{theo:weak}.
In Section~\ref{sect:wave}, we generalize the arguments for the stochastic wave equation~\eqref{eq:SPDE_wave_intro}, see Theorem~\ref{theo:wave}. Finally, in Section~\ref{sect:non_comm} we generalize the approach in the case when $A$ and $Q$ do not commute: we provide a general result, Theorem~\ref{theo:weak_gen}, and identify weak rates of convergence in a specific case, Theorem~\ref{theo:weak_special}.

\section{Notations and Assumptions}\label{sect:Not_Ass}

In this section
we introduce sufficient assumptions to ensure well-posedness of the 
following SPDE, written in abstract form:
\begin{equation}\label{eq:SPDE}
du(t)=Au(t)dt+F\left(u(t)\right)dt+dW^Q(t), \; u(0)=u_0.
\end{equation}
We work with assumptions on $A$ which render the equation in semilinear
parabolic form.  We state the precise definitions and assumptions on $A$
and $F$ below, following the standard setting of \cite{DPZ}. These 
assumptions will also be used in Section~\ref{sect:wave}, where we 
study a semilinear damped-driven wave equation~\eqref{eq:SPDE_wave},
also constructed from $A$ and $F$. 

We describe the function space setting, the assumptions on $A$,
the assumptions on $F$ and the assumptions on $Q$, in turn,
in the following subsections.

\subsection{The function space}

Our state space is a separable, infinite-dimen\-sional, real Hilbert space $H$, equipped 
with its scalar product $\langle\cdot,\cdot\rangle_H$ and norm $|\cdot|_H$. We also use 
the notations $\langle\cdot,\cdot\rangle$ and $|\cdot|$ when no confusion is 
likely to arise.
A typical example to keep in mind is $H=L^2(\mathcal{D})$ where $\mathcal{D}\subset\R^d$ 
is a bounded, open set with smooth boundary.

\subsection{The linear operator}

The operator $A$ appearing in \eqref{eq:SPDE} is an unbounded linear operator on the Hilbert space $H$.

\begin{hyp}\label{ass:A}
The linear operator $A$ is defined on a dense domain $D(A)\subset H$, with values in $H$. 
It is self-adjoint with compact resolvent, such that there exists a non-decreasing sequence of 
strictly positive real numbers $(\lambda_k)_{k\in\N^{*}}$, and a corresponding orthonormal basis 
$(e_k)_{k\in\N^{*}}$ of $H$ such that
$Ae_k=-\lambda_ke_k$ for all $k\in \N^{*}$.
\end{hyp}
Here, we wrote $\N^{*}=\left\{1,2,\ldots\right\}$. The fact that $A$ has compact resolvent
implies that it is unbounded, in particular $\lim_{k \to \infty}\lambda_k = \infty$.
Under Assumptions \ref{ass:A}, 
the operator $A$ generates an analytic semigroup on $H$ denoted by  $(e^{tA})_{t\geq 0}$: for any $t\geq 0$
$$e^{tA}u=\sum_{k\in\N^*}e^{-t\lambda_k}\langle u,e_k\rangle e_k\;.$$
This semigroup enjoys good regularization properties, which we recall in 
Proposition~\ref{prop:regul}. To state them, we define the interpolation spaces 
$\HH^\s$ for $\s \in \R$ as the closure of the linear span of the $e_k$ under the norm
\begin{equ}
|u|_{\s}^2 = \sum_{k=1}^{+\infty}\lambda_{k}^{\s}\langle u,e_{k}\rangle^{2}\;.
\end{equ}
In particular, one has $\HH^0=H$ and $\HH^2=D(A)$.
It is immediate that, for every $\s \in \R$, 
$(-A)^{\s}$ is a bounded linear operator from $\HH^{2\s}$ to $H$.
Moreover, for any $\s\in \R$, $\HH^\s$ is a Hilbert space with scalar 
product $\langle \cdot,\cdot \rangle_{\s}$, where for any $x,y\in\HH^{\s}$
$$\langle x,y\rangle_{\s}=\langle (-A)^{\s/2}x,(-A)^{\s/2}y\rangle.$$

\begin{ex} 
A typical example is the Laplace operator in $\mathcal{D}=(0,1)$, complemented with homogeneous Dirichlet boundary conditions. In this case, $\lambda_k=\pi^2k^2$ and $e_k(\xi)=\sqrt{2}\sin(k\pi \xi)$.
Then $\HH^1$ is the space $H^1_0(D).$
\end{ex}

We can now state the semigroup regularization properties.
\begin{propo}\label{prop:regul}
For any $-1\leq\s_1\leq \s_2\leq 1$, there exists $C_{\s_1,\s_2}\in(0,+\infty)$ such that:
\begin{enumerate}
\item for any $t>0$ and $u\in \HH^{\s_1}$
$$|e^{tA}u|_{\s_2}\leq C_{\s_1,\s_2} t^{-(\s_2-\s_1)/2}|u|_{\s_1};$$
\item for any $0<t_1<t_2$ and $u\in \HH^{\s_2}$
$$|e^{t_2A}u-e^{t_1A}u|_{\s_1}\leq C_{\s_1,\s_2}(t_2-t_1)^{(\s_2-\s_1)/2}|u|_{\s_2}.$$
\item for any $0<t_1<t_2$ and $u\in \HH^{\s_1}$
$$|e^{t_2A}u-e^{t_1A}u|_{\s_1}\leq C_{\s_1,\s_2}\frac{(t_2-t_1)^{(\s_2-\s_1)/2}}{t_{1}^{(\s_2-\s_1)/2}}|u|_{\s_1}.$$
\end{enumerate}
\end{propo}
We omit the proof of this classical result; for instance see \cite[Chapter 2, Theorem 6.13]{Pazy} for a general statement (analytic semigroups).

\subsection{The nonlinearity}

We consider a function $F$ defined on $\HH^{\s_F}$ for some nonnegative regularity parameter $\s_F$, and taking values in $\HH^{-\s_F}$.
\begin{hyp}\label{ass:F}
For some $\s_F\in [0,1)$, the map $F:\HH^{\s_F}\rightarrow \HH^{-\s_F}$ is of class $\mathcal{C}^2$, with bounded first and second-order Fr\'echet derivatives.
\end{hyp}

Equivalently, $F_{\s_F}:=A^{-\s_F/2}\circ F\circ A^{-\s_F/2}:H\rightarrow H$ is a function of class $\mathcal{C}^2$, with bounded derivatives. Due to the inclusion property of the spaces $\HH^\s$, it is natural to take $\s_F\geq 0$ as small as possible in Assumption \ref{ass:F}: if the regularity and boundedness conditions of Assumption \ref{ass:F} are satisfied for some $\s_F$, they are also satisfied for all $\s>\s_F$.

\begin{rem}
More generally, we can consider functions $F:\HH^{\s_{F}^{1}}\rightarrow \HH^{-\s_{F}^{2}}$ for $\s_{F}^{1}\neq\s_{F}^{2}$, and $\s_{F}^{1},\s_{F}^{2}\geq 0$. In this setting, it is natural to take the smallest possible values for both $\s_{F}^{1}$ and $\s_{F}^{2}$. However, thanks to a simple shift in the definition of the spaces $\HH^{\s}$, it is always 
possible to fit into the framework of Assumption \ref{ass:F} by redefining $H$.
\end{rem}

\begin{ex}
If $\Psi:\HH^{\s}\rightarrow \R$ is a function of class $\mathcal{C}^3$ with bounded derivatives, set $F=-D\Psi$. Then $F(u)\in\HH^{-\s}$, for any $u\in\HH^{\s}$, thanks to the natural identification of the dual space of $\HH^{\s}$ with $\HH^{-\s}$. When $\s>0$, the potential function $\Psi$ may thus only be defined on a strict subspace $\HH^{\s}$ of $\HH$.
\end{ex}

\subsection{The noise term}

The Stochastic PDE \eqref{eq:SPDE} is driven by an additive noise which is white in time, and can be either white (when $Q$ is the identity mapping), or colored in space.

The covariance operator $Q$ is assumed to satisfy the following conditions.
\begin{hyp}\label{ass:B}
The linear operator $Q:H\rightarrow H$ is self-adjoint, bounded, and there exists a bounded sequence $(q_k)_{k\in \N^*}$ of nonnegative real numbers such that for any $k\in \N^*$
$$Qe_k=q_ke_k,$$
where $(e_k)_{k\in \N^*}$ is as in Assumption \ref{ass:A}.
\end{hyp}

We do not require that $Q$ is trace class. Assumption~\ref{ass:B} implies that $A$ and $Q$ commute, which may be restrictive in practice. Commutativity 
is not required to obtain weak error estimates at a fixed time $t$, 
but we use it here to prove, under quite general conditions,
that the weak order of convergence for the trajectories 
is twice the strong order. In Section \ref{sect:non_comm}, we 
exhibit an example where one can prove such a result even though 
$A$ and $Q$ do not commute, but the general non-commuting case will
remain open.

We now recall the definition of the cylindrical Wiener process $W$, and of the associated $Q$-cylindrical Wiener process $W^Q$.
\begin{hyp}\label{ass:W}
$W$ is a cylindrical Wiener process on $H$, on a probability space $(\Omega,\mathcal{F},(\mathcal{F}_t)_{t\in\R^+},\PP)$, where the filtration $(\mathcal{F}_t)_{t\in\R^+}$ satisfies the usual conditions: for any $t\geq 0$
\begin{equation}\label{expansion_W}
W(t)=\sum_{k\in\N^*}\beta_{k}(t)e_k,
\end{equation}
where $(\beta_k)_{k\in\N^*}$ is a sequence of independent, standard, real-valued Wiener processes (with respect to the filtration $(\mathcal{F}_{t})_{t\in\R^+}$) and $(e_k)_{k\in \N^*}$ is the complete orthonormal system of $H$ introduced in Assumption \ref{ass:A}.

The series \eqref{expansion_W} converges in any Hilbert space $\tilde{H}$ such that $H$ is contained in $\tilde{H}$ with a Hilbert-Schmidt embedding mapping. The resulting process  depends neither on the choice of the complete orthonormal system, nor on the choice of
$\tilde H$ (modulo canonical embeddings).

Similarly, we define the $Q$-cylindrical Wiener process $W^Q(\cdot)$; since we do not assume in general 
that $Q$ is trace-class, $W^Q$ 
also takes values in some larger space $\tilde{H}$, containing $H$, as does
$W$. For any $t\geq 0$, set
\begin{equation}\label{expansion_W^Q}
W^Q(t)=\sum_{k\in\N^*}\sqrt{q_k}\beta_{k}(t)e_k.
\end{equation}
\end{hyp}

One of the main conditions we will require is that the stochastic convolution $W^{A,Q}$ (the solution of~\eqref{eq:SPDE} when $F=0$) takes values in $\HH^{\s_F}$. Assumptions on the covariance operator $Q$ will be made precise in Section~\ref{sect:SPDE_Itomap}, in particular see Proposition~\ref{propo:Link}.

In the sequel, we use the notation $\Tr\bigl(L\bigr)$ for the trace of a trace-class non-negative symmetric linear operator $L\in\mathcal{L}(H)$:
$$\Tr(L)=\sum_{n\in\N^*}\langle L\epsilon_n,\epsilon_n\rangle<+\infty,$$
for some (and therefore all) complete orthonormal basis $(\epsilon_n)_{n\in\N^*}$ of $H$.

\section{Well-posedness and the It\^o Map}\label{sect:Ito_parabolic}

In this section, we state well-posedness results for the Stochastic PDE \eqref{eq:SPDE}, and we define one of the essential tools for our study of the weak error on trajectories: the It\^o map. This application allows us to express the solution of the semilinear equation as a (deterministic) function of the solution of the linear equation ($F=0$) with the same noise perturbation.
Most of the material in this section is classical, but explicit
inclusion of the important definitions and properties will help
understanding of the weak trajectory error analysis in this paper.

For any regularity parameter $\s\in\R$ and any $T\in(0,+\infty)$, we introduce the space of trajectories
$$
\mathcal{C}_{\s,T}=\mathcal{C}([0,T],\HH^\s),
$$
the space of continuous functions of the time variable, with values in the Hilbert space $\HH^\s$. Elements of the space $\mathcal{C}_{\s,T}$ are referred to as trajectories in the sequel.
We also define the family of supremum norms $\|\cdot\|_{\infty,\s,T}$:
$$
\|\mathcal{X}\|_{\infty,\s,T}=\sup_{0\leq t\leq T}|\mathcal{X}(t)|_{\s}
$$
for any $\mathcal{X}\in\mathcal{C}_{\s,T}$; endowed with the associated topology, $\mathcal{C}_{\s,T}$ is a separable Banach space. 

In the subsequent subsections we study the deterministic problem arising
when $Q=0$, we define the It\^o map and we study the SPDE through the
It\^o map. 

\subsection{The deterministic semilinear PDE}

Under the global Lipschitz condition on $F:\HH^{\s_F}\rightarrow \HH^{-\s_F}$, from Assumption~\ref{ass:F}, the well-posedness in terms of mild solutions of the deterministic semi-linear equation
\begin{equation}\label{eqdeter}
\frac{d\U(t)}{dt}=A\U(t)+F(\U(t)),\; \U(0)=u_0\in\HH^\s.
\end{equation}
is a standard result, using regularization properties of the semigroup $\bigl(e^{tA}\bigr)_{t\in\R^+}$ (see Proposition~\ref{prop:regul}) and a Picard iteration argument. 

\begin{propo}\label{prop:eq_deter}
Assume $\s_{F}\leq \s<1$, and that Assumption \ref{ass:A} and \ref{ass:F} are satisfied. 
For any initial condition $u_0\in\HH^{\s}$, and any time $T\in(0,+\infty)$, there exists a unique mild solution of equation \eqref{eqdeter}, satisfying:
\begin{itemize}
\item $\U(0)=u_0$;
\item for any $t\in[0,T]$, $\U(t)\in\HH^{\s}$, and $t\mapsto \U(t)\in\mathcal{C}_{\s,T}$;
\item for any $t\geq 0$,
\begin{equation}\label{eq:mild_Y_def}
\U(t)=e^{tA}u_0+\int_{0}^{t}e^{(t-r)A}F(\U(r))dr.
\end{equation}
\end{itemize}
\end{propo}

Note that for $\s_F\leq \s<1$, $F:\HH^{\s}\rightarrow\HH^{-\s_{F}}$ is globally Lipschitz continuous, thanks to Assumption~\ref{ass:F}.



\subsection{The It\^o map}\label{sect:Itomap}

The result of Proposition \ref{prop:eq_deter} is extended in a straightforward manner (using the Picard iteration argument) to the case where some perturbation by some function $w\in\mathcal{C}_{\s,T}$ is added to the mild formulation \eqref{eq:mild_Y_def}. 
\begin{propo}\label{prop:eq_deter+w}
Assume $\s_{F}\leq \s<1$, and that Assumption \ref{ass:A} and \ref{ass:F} are satisfied. Let $u_0\in\HH^{\s_0}$ for $\s_0\geq \s$, let $T\in(0,+\infty)$,
and let $w\in\mathcal{C}_{\s,T}$. 
Then there exists a unique function $\U^w\in\mathcal{C}_{\s,T}$ such that for any $t\geq 0$,
\begin{equation}\label{eq:Y^w}
\U^w(t)=e^{tA}u_0+\int_{0}^{t}e^{(t-s)A}F(\U^w(r)) dr+w(t).
\end{equation}
\end{propo}

We can now define the It\^o map.
\begin{defi}
The It\^o map associated with the SPDE \eqref{eq:SPDE} is defined as the 
map $\Theta\colon w\in\mathcal{C}_{\s,T}\mapsto \U^w$, with $\U^w$ given by~\eqref{eq:Y^w}.
\end{defi}

The It\^o map $\Theta$ depends on the regularity parameter $\s$, the linear operator $A$, 
the nonlinear coefficient $F$, the initial condition $x_0$ and the time $T\in(0,+\infty)$. However, to lighten the notation, we do not mention explicitly these dependences in the sequel.
The It\^o map inherits the regularity properties of $F$ from Assumption \ref{ass:F}.

\begin{theo}\label{theo:Itomap}
The It\^o map $\Theta$ is of class $\mathcal{C}^2$ on the Banach space $\mathcal{C}_{\s,T}=\mathcal{C}([0,T],\HH^\s)$, with bounded Fr\'echet derivatives of first and second order.
\end{theo}

We only give a sketch of proof (see~\cite{PillaiStuartThiery:14} for details on the continuity of $\Theta$). Theorem~\ref{theo:Itomap} is a consequence of the Implicit Function Theorem. Indeed, the mappings $(y,w)\in\mathcal{C}_{\s,T}\times \mathcal{C}_{\s,T}\mapsto (F\circ y,w)\in \mathcal{C}_{-\s,T}\times \mathcal{C}_{\s,T}$ and $(z,w)\in \mathcal{C}_{-\s,T}\times \mathcal{C}_{\s,T} \mapsto Y\in \mathcal{C}_{\s,T}$, such that $Y(t)=\int_{0}^{t}e^{(t-r)A}z(r)dr+w(t)$, are of class $\mathcal{C}^2$.

\subsection{The SPDE and the It\^o map}\label{sect:SPDE_Itomap}

The study of the well-posedness of the SPDE \eqref{eq:SPDE} is done in two steps. First we consider the linear case with additive noise, i.e. when $F$ is identically $0$ and the initial condition is $x_0=0$; the unique mild solution is the so-called stochastic convolution, for which we give below the necessary properties concerning spatial regularity. Second, we consider the full semi-linear equation \eqref{eq:SPDE}, and since the noise is additive we use the It\^o map to define solutions, in a strong sense with respect to the probability space. The material is standard, see~\cite{DPZ}.

We first define an important regularity parameter $\s_Q$, depending on $Q$.
\begin{defi}\label{defi:s_Q}
Assume that $\Tr((-A)^{-1}Q)<+\infty$, and introduce
\begin{equation}\label{eq:s_Q}
\s_Q^0=\sup\left\{ \s\in\R^+; \Tr((-A)^{\s-1}Q)<+\infty\right\}\geq 0.
\end{equation}
We also set $\s_Q=\min(s_Q^0,1)$.
\end{defi}

\begin{propo}
The linear stochastic equation with additive noise,
\begin{equation}\label{eqstolin}
dZ(t)=AZ(t)dt+dW^Q(t),\; Z(0)=0,
\end{equation}
admits a unique mild solution $Z\in\mathcal{C}([0,T],\HH^\s)$, for any $\s<\s_Q$, and any $T\in(0,+\infty)$; this is the $\HH^\s$-valued process such that for any $0\leq t\leq T$
$$Z(t)=\int_{0}^{t}e^{(t-r)A}dW^Q(r).$$
This process is denoted by $W^{A,Q}$ and is called the stochastic convolution.
\end{propo}

For the ease of the exposition, we define a set of admissible parameters $\s$.
\begin{defi}\label{defi:admiss}
A parameter $\s\in\R^+$ is called \emph{admissible} if it satisfies $\s_F\leq \s<1$ 
and $\s< s_Q$. 
\end{defi}

The set of admissible parameters is of course non empty if and only if
$\s_F< \s_Q$, so from now on we assume this is the case. Note that $\s=\s_F$ is an admissible parameter.

\begin{propo}\label{propo:Link}
Let $T\in(0,+\infty)$ and assume that $u_0\in\HH^{\s}$
for some admissible $\s$. Then \eqref{eq:SPDE} admits a unique mild solution $X$, {\it i.e.} an $\HH^{\s}$-valued process such that, for any $0\leq t\leq T$,
\begin{equation}\label{mildsto}
u(t)=e^{tA}u_0+\int_{0}^{t}e^{(t-r)A}F(u(r))dr+\int_{0}^{t}e^{(t-r)A}dW^Q(r).
\end{equation}
Moreover, $u$ admits a version in $\mathcal{C}([0,T],\HH^\s)$ such that,
denoting by $\Theta$ the It\^o map, we have
\begin{equation}\label{eq:Ito_inf}
u=\Theta(W^{A,Q})\;.
\end{equation}
\end{propo}

\begin{proof}
Once we know that $W^{A,Q}$ admits a version in $\mathcal{C}([0,T],\HH^\s)$,
we can define $X$ by \eqref{eq:Ito_inf} and verify that it solves \eqref{mildsto}.
Uniqueness of the mild solution (modulo indistinguishability of stochastic processes)
follows from a simple Picard iteration argument.
The existence of a version of $W^{A,Q}$ in $\mathcal{C}([0,T],\HH^\s)$, with $\E\|W^{A,Q}\|_{\infty,\s,T}<+\infty$, follows from \cite[Theorem~5.9]{DPZ} and the admissibility of $\s$. 
\end{proof}

\section{Spectral Galerkin Discretization of parabolic SPDEs}\label{sect:Galerkin_parabolic}

In this section we introduce the spectral Galerkin approximation of the 
SPDE~\eqref{eq:SPDE} and we study strong and weak error estimates.
Section~\ref{sect:not_Galerkin} defines the discretization scheme,
Section~\ref{sect:strong_parabolic} contains the strong convergence result,
and Section~\ref{sect:weak_parabolic} the weak convergence result.
We discuss the two results in Section~\ref{sect:comments}.

\subsection{Definition of the discretization scheme}\label{sect:not_Galerkin}

We approximate the solution $u$ of the SPDE \eqref{eq:SPDE} by a projection onto the 
finite dimensional subspace $H_{N} \subset H$
spanned by $\{e_1,\ldots,e_N\}$,
with the $e_i$ as in Assumption~\ref{ass:A}.
To this end define $P_N\in\mathcal{L}(H)$ as the orthogonal projection from $H$ onto $H_N$, where $\mathcal{L}(H)$ is the space of bounded linear operators from $H$ to $H$. In the sequel, the identity mapping on $H$ is denoted by $I\in\mathcal{L}(H)$.
Set also $H_{N}^{\perp}=\text{span}\left\{e_n; n\geq N+1\right\}$, 
and $P_{N}^{\perp}=I-P_N$ the associated orthogonal projection. 
For any $\s\in\R$, $H_N$ is a subspace of $\HH^\s$ and we can view $P_N$ as an element
of $\mathcal{L}(\HH^\s)$, which is still an orthogonal projection operator.

Given $N\in\N^*$, the process $u_N$ with values in $H_N$, 
is defined as the unique mild solution of the SPDE
\begin{equation}\label{eq:SPDE_projN}
du_N(t)=Au_N(t)dt+F_N\left(u_N(t)\right)dt+P_NdW^Q(t), \; u_N(0)=P_Nu_0.
\end{equation}
where $F_N=P_N\circ F:\HH^{s_F}\rightarrow\HH^{-s_F}$ satisfies Assumption~\ref{ass:F}, with bounds on $F_N$ and its derivatives which are uniform with respect to $N$. Note that $u_N$ satisfies the identity (mild formulation of~\eqref{eq:SPDE_projN})
\begin{align*}
u_N(t)&= P_Ne^{tA}u_0+P_N\int_{0}^{t}e^{(t-r)A}F(u_N(r))dr+P_NW^{A,Q}(t)\;.
\end{align*}

As a consequence, we make the following simple observation, which is crucial to obtain the strong and weak error estimates in $\mathcal{C}_{\s,T}$.
\begin{theo}\label{theo:Ito_N}
Let $R_N$ be given by
\begin{equation}\label{eq:R_N}
R_N(t)=(P_N-I)e^{tA}u_0+\int_{0}^{t}(P_N-I)e^{(t-r)A}F(u_N(r))dr.
\end{equation}
Then, one has the identity
\begin{equation}\label{eq:Ito_N}
u_N=\Theta(P_NW^{A,Q}+R_N)\;.
\end{equation}
\end{theo}

Moreover, one has the following a priori estimate.
\begin{lemme}\label{lem:a_priori_bound_s}
For any $T>0$, any $u_0\in \HH^{\s_F}$, there exists $C(T,|u_0|_{\s_F})\in(0,+\infty)$ such that
\begin{equation}\label{eq:a_priori_bound_s}
\sup_{N\in\N^*}\E \big\|u_N\big\|_{\infty,\s_F,T}\leq C(T,|u_0|_{\s_F}).
\end{equation}
\end{lemme}

\begin{proof}
Note first that $P_NW^{A,Q} = W^{A,Q_N}$ with $Q_N = Q P_N$.
It then suffices to note that 
the bounds obtained by the Picard iteration \eqref{mildsto} only depend on the 
Lipschitz constant of $F$, on $\lambda_1$, and on the exponents
$\s_F$ and $\s_Q$. All of these quantities can be chosen independent of $N$
when $F$ is replaced by $F_N$ and $Q$ by $Q_N$.
\end{proof}


\subsection{Strong convergence}\label{sect:strong_parabolic}

Our first result is a strong error estimate on trajectories in $\mathcal{C}_{\s,T}$. 
Given $T\in(0,+\infty)$, we give a bound on the expectation of the 
$\|\cdot\|_{\s,T}$-norm of the difference between $u$ and $u_N$.
For completeness we include a detailed proof, even though the main focus
of our work is weak error and hence, in this section, Theorem \ref{theo:weak}; moreover, some bounds obtained during the proof are used later in the proof of the weak error estimates. The stochastic part is controlled thanks to the factorization method (see \cite[p.~128]{DPZ}).

\begin{theo}\label{theo:strong}
Let $T\in(0,+\infty)$, assume that Assumptions \ref{ass:A}, \ref{ass:F} and \ref{ass:B}
hold, and let $u_0\in\HH^{\s_0}$ for $\s_0 \ge \s$ with $\s$ admissible. Then for any $\epsilon\in(0,\s_Q-\s)$, there exists a constant $C_{\epsilon,\s}\in(0,+\infty)$, such that for any $N\in\N^*$
\begin{equation}\label{eq:strong}
\E\|u-u_N\|_{\infty,\s,T}\leq C_{\epsilon,\s}\Bigl(\frac{1}{\lambda_{N+1}^{(\s_0-\s)/2}}|u_0|_{\s_0}+\frac{1}{\lambda_{N+1}^{1-(\s_F+\s+\epsilon)/2}}+\frac{1}{\lambda_{N+1}^{(\s_Q-\s-\epsilon)/2}}\Bigr).
\end{equation}
\end{theo}

\begin{proof}
To simplify the notation, $C$ denotes a real number $(0,+\infty)$ which does not depend on $N$, on $x_0$ and on $F$. It may vary from line to line.
Since $\Theta$ is Lipschitz thanks to Theorem~\ref{theo:Itomap}, using \eqref{eq:Ito_inf} and \eqref{eq:Ito_N}, only the quantities $\E\|R_N\|_{\infty,\s,T}$ and $\E\|W^{A,Q}-P_NW^{A,Q}\|_{\infty,\s,T}^{2}$ need to be controlled.
First, for any $t\in[0,T]$,
\begin{eqnarray*}
|R_N(t)|_{\s}&\leq& |P_N^\perp e^{tA}u_0|_{\s}+\int_{0}^{t}\big|P_N^\perp e^{(t-r)A}F(u_N(r))\big|_{\s}dr\\ 
&\leq& |(-A)^{\s/2}P_N^\perp e^{tA}(-A)^{-\s_0/2}|_{\mathcal{L}(H)}|u_0|_{\s_0}\\
&~&+ \int_{0}^{t}|(-A)^{\s/2}P_N^\perp e^{(t-r)A}(-A)^{\s_F/2}|_{\mathcal{L}(H)}|F(u_N(r))|_{-\s_F}dr\\
&\leq& C\lambda_{N+1}^{(\s-\s_0)/2}|u_0|_{\s_0}\\
&~&+C\bigl(1+\|u_N\|_{\infty,\s,T}\bigr)(-A)^{-1+\epsilon/2+\s/2+\s_F/2}P_N^\perp|_{\mathcal{L}(H)}\int_{0}^{t}\frac{C}{(t-r)^{1-\epsilon/2}}dr,
\end{eqnarray*}
thanks to the Lipschitz continuity of $F:\HH^{\s_F}\to\HH^{-\s_F}$ (Assumption~\ref{ass:F}).

Using the a priori bound \eqref{eq:a_priori_bound_s} of Lemma~\ref{lem:a_priori_bound_s}, we thus obtain 
\begin{equation}\label{eq:proof_strong_deter}
\E\|R_N\|_{\infty,\s,T}\leq C\lambda_{N+1}^{-(\s_0-\s)/2}|u_0|_{\s_0}+C\lambda_{N+1}^{-1+(\s_F+\s+\epsilon)/2}.
\end{equation}
It remains to deal with the contribution of the discretization of the stochastic convolution. As in \cite{DPZ}, we
can write $W^{A,Q}-P_NW^{A,Q}=\Gamma\bigl(P_N^\perp Z^{A,Q}\bigr)$, where 
\begin{equation}\label{eq:def_Gamma}
\Gamma(z)(t)=\int_{0}^{t}(t-r)^{\epsilon-1}e^{(t-r)A}z(r)dr\;,
\end{equation}
$Z^{A,Q}$ is the auxiliary process given by
\begin{equation}\label{eq:aux_Z}
Z^{A,Q}(t)=\int_{0}^{t}(t-r)^{-\epsilon}e^{(t-r)A}dW^Q(r)\;,
\end{equation}
and $\epsilon \in (0,1)$.
Since, as in \cite[p.~128]{DPZ}, $\Gamma$ maps 
$L^p([0,T],\HH^{\s})\rightarrow \mathcal{C}([0,T],\HH^{\s})$ for $p>1$ sufficiently 
large (depending on $\epsilon$), we have
\begin{equation}\label{eq:proof_strong_sto}
\begin{aligned}
\E\|W^{A,Q}&-P_NW^{A,Q}\|_{\infty,\s,T}^{p} \leq C_{p,T}\left(\int_{0}^{T} \E |P_N^\perp Z^{A,Q}(r)|_{\s}^{p} dr \right)\\
&\leq C_{p,T}|P_N^\perp (-A)^{(\s-\s_Q+\epsilon)/2}|_{\mathcal{L}(H)}^{p}\left(\int_{0}^{T}\E|Z^{A,Q}(r)|_{\s_Q-\epsilon}^{p} dr \right)\\
&\leq C\lambda_{N+1}^{-\frac{p(\s_Q-\s-\epsilon)}{2}},
\end{aligned}
\end{equation}
by Fernique's theorem, since $\sup_{0\leq t\leq T}\E|Z^{A,Q}(t)|_{\s_Q-\epsilon}^{2}<+\infty$ as soon as $\epsilon\in(0,\s_Q)$. 
\end{proof}

\subsection{Weak convergence}\label{sect:weak_parabolic}

In this section, we state and prove our main result, Theorem~\ref{theo:weak}, which is a weak 
error estimate in the space of trajectories $\mathcal{C}_{\s,T}$. Considering the contribution of the stochastic parts, the weak order of 
convergence is twice the strong order appearing in Theorem~\ref{theo:strong}. 
For this, we need an appropriate set of test functions to define a metric on the set 
of probability distributions on $\mathcal{C}_{\s,T}$, which is the purpose of
the following definition.


\begin{defi}\label{defi:test_funct}
Let $T\in(0,+\infty)$ and $\s$ be an admissible regularity parameter in the sense of Definition~\ref{defi:admiss}. A function $\Phi:\mathcal{C}_{\s,T}\rightarrow \R$ is called an \emph{admissible test function} if it is bounded and of class $\mathcal{C}^2$, with bounded Fr\'echet derivatives of first and second order, where the metric on $\mathcal{C}_{\s,T}$ is induced by the norm $\|\cdot\|_{\infty,\s,T}$. 
\end{defi}

This class is somewhat restrictive since we require $\mathcal{C}^2$ regularity in order to 
be able to perform a second-order Taylor expansion of the error. However, some interesting 
observables depending on the whole trajectory and falling in this class are 
now given: 

\begin{ex}\label{ex:examples_admiss_test_funct}
Let $\phi:\HH^\s\rightarrow \R$ be bounded and of class $\mathcal{C}^2$ with bounded first and second order derivatives. Then for any $t\in[0,T]$,
$$\Phi_t: \U\in\mathcal{C}_{\s,T}\mapsto \phi(\U(t))$$
is an admissible test function. Moreover, for any $0\leq t_1<t_2\leq T$,
$$\Phi_{t_1,t_2}: \U\in\mathcal{C}_{\s,T}\mapsto \int_{t_1}^{t_2}\phi(\U(t))dt$$
is another admissible test function. Finally, if $\Psi:\R\rightarrow \R$ is of class $\mathcal{C}^2$, with bounded first and second order derivatives, the mapping $\Psi\circ \Phi_{t_1,t_2}$ is also an admissible test function.
\end{ex}

The main object of study in this section is the weak error
\begin{equation}\label{eq:e_N_Phi}
e_N(\Phi,\s)=\E[\Phi(u)]-\E[\Phi(u_N)]\;,
\end{equation}
where $\s$ is an admissible regularity parameter, $\Phi$ is an admissible test function from $\mathcal{C}_{\s,T}$ to $\R$, $u$ is the solution of the SPDE \eqref{eq:SPDE} and $u_N$ is the approximation in dimension $N$ given by \eqref{eq:SPDE_projN}. To simplify the notation, we fix the time $T\in(0,+\infty)$ and do not mention the dependence of $e_N(\Phi,\s)$ with respect to this quantity.
Our main result is the following Theorem~\ref{theo:weak};
comments on this theorem and its relation to Theorem~\ref{theo:strong}, 
are given in Section~\ref{sect:comments}.

\begin{theo}\label{theo:weak}
Let $T\in(0,+\infty)$, let Assumptions \ref{ass:A}, \ref{ass:F} and \ref{ass:B} hold,
let $u_0\in\HH^{\s_0}$ for $\s_0 \ge \s$ with $\s$  an admissible parameter,
and let $\Phi:\mathcal{C}_{\s,T}\rightarrow \R$  be an admissible test function.
Then for any $\epsilon\in(0,\s_Q-\s)$, there exists a constant $C_{\epsilon,\s}(\Phi)\in(0,+\infty)$, such that for any $N\in\N^*$
\begin{equation}\label{eq:weak}
\big|e_N(\Phi,\s)\big|\leq C_{\epsilon,\s}(\Phi)\Bigl(\frac{1}{\lambda_{N+1}^{(\s_0-\s)/2}}|u_0|_{\s_0}+\frac{1}{\lambda_{N+1}^{1-(\s_F+\s+\epsilon)/2}}+\frac{1}{\lambda_{N+1}^{\s_Q-\s-\epsilon}}\Bigr).
\end{equation}
\end{theo}

\begin{proof}
Thanks to the definition of the It\^o map $\Theta$ and Proposition \ref{propo:Link}, we have $u=\Theta(W^{A,Q})$; moreover $u_N=\Theta(P_NW^{A,Q}+R_N)$, for any $N\in \N^*$, by Theorem \ref{theo:Ito_N}.
Therefore, setting $\Psi = \Phi \circ \Theta$ for the It\^o map $\Theta$,
the weak error \eqref{eq:e_N_Phi} can be rewritten as
\begin{align*}
e_N(\Phi,\s)&=\E[\Phi(u)]-\E[\Phi(u_N)]\\
&=\E[\Phi\circ\Theta(W^{A,Q})]-\E[\Phi\circ\Theta(P_NW^{A,Q}+R_N)]\\
&=\E[\Psi(W^{A,Q})]-\E[\Psi(P_NW^{A,Q})]\\
&~+\E[\Psi(P_NW^{A,Q})]-\E[\Psi(P_NW^{A,Q}+R_N)].
\end{align*}
Thanks to Theorem \ref{theo:Itomap}, the map $\Psi:\mathcal{C}_{\s,T}\rightarrow \R$ is 
again an admissible test function (in the sense of Definition \ref{defi:test_funct}); in particular, it is Lipschitz continuous, and for any $N\in\N^*$
\begin{equs}
\Big|\E[\Psi&(P_NW^{A,Q})]-\E[\Psi(P_NW^{A,Q}+R_N)]\Big| \leq C_{\s}^{1}(\Phi)\E\|R_N\|_{\infty,\s,T}\\
&\leq \frac{C_{\s}^{1}(\Phi)}{\lambda_{N+1}^{(\s_0-\s)/2}}|u_0|_{\s_0}+\frac{C_{\s}^{1}(\Phi)}{\lambda_{N+1}^{1-(\s_F+\s+\epsilon)/2}},
\end{equs}
for some $C_{\s}^{1}(\Phi)\in(0,+\infty)$, thanks to 
the component of the strong error derived as \eqref{eq:proof_strong_deter}.


It remains to study the part of the error due to the discretization of the stochastic convolution $W^{A,Q}$.
The test function $\Psi$ being admissible, it is of class $\mathcal{C}^2$ with bounded first and second order derivatives, so that
\begin{equs}
|\Psi(&W^{A,Q})-\Psi(P_NW^{A,Q})- D\Psi(P_NW^{A,Q}).\bigl[P_N^\perp W^{A,Q}\bigr]|\\
&\le C \|P_N^\perp W^{A,Q}\|_{\infty,\s,T}^{2}\;, \label{e:boundweak}
\end{equs}
for some constant $C$.
The expectation of this term is easily controlled by using the strong error estimate proved above: for any $\epsilon\in(0,\s_Q-\s)$, there exists $C$ such that for any $N\in\N^*$
\begin{equ}[e:boundstrong]
\E \|P_N^\perp W^{A,Q}\|_{\infty,\s,T}^{2} \leq \frac{C}{\lambda_{N+1}^{\s_Q-\s-\epsilon}}\;,
\end{equ}
thanks to \eqref{eq:proof_strong_sto}.

To control the first order term, the key observation is that that the $\mathcal{C}_{\s,T}$-valued random variables $P_NW^{A,Q}$ and $P_N^\perp W^{A,Q}$ are independent
since the former depends only on $\{\beta_i\,:\, i \le N\}$, while the latter
only depends on $\{\beta_i\,:\, i > N\}$. Furthermore, $P_N^\perp W^{A,Q}$ is a centred
Gaussian process, so that
\begin{equ}
\E \bigl( D\Psi(P_NW^{A,Q}).\bigl[P_N^\perp W^{A,Q}\bigr]\bigr) = 0\;.
\end{equ}
Combining this with \eqref{e:boundweak} and \eqref{e:boundstrong}, we obtain the weak error estimate
$$\Big|\E[\Psi(W^{A,Q})]-\E[\Psi(P_NW^{A,Q})]\Big|\leq \frac{C_{\epsilon,\s}(\Phi)}{\lambda_{N+1}^{\s_Q-\s-\epsilon}}\;,$$
thus concluding the proof.
\end{proof}

\subsection{Comments on Theorems~\ref{theo:strong} and~\ref{theo:weak}}\label{sect:comments}

The aim of this section is to compare the strong and weak orders of convergence from Theorems~\ref{theo:strong} and \ref{theo:weak}. For simplicity of the discussion, we take $\epsilon=0$, even if the results are only valid for $\epsilon>0$. 
Note that $\s=\s_F$ is an admissible parameter, and that it is a natural choice.

When $\s$ increases, the orders of convergence of each term in~\eqref{eq:strong} and~\eqref{eq:weak} decreases; observe that the decrease is slower for the third term in~\eqref{eq:weak} than for the other terms (one finds $\s$ instead of $\s/2$): indeed that terms comes from the second-order term in the Taylor expansion. On the contrary, the higher the spatial regularity of the initial condition $u_0$ (increase of $\s_0$), of the coefficient $F$ (decrease of $\s_F$) and of the covariance operator ($\s_Q$ increases), the higher the rates of convergence.

The two first error terms in~\eqref{eq:strong} and~\eqref{eq:weak} are the same. These terms are due to the discretization of the initial condition $x_0$ and of the coefficient $F$. Indeed, in the weak error estimate we have only used the Lipschitz continuity of the It\^o map to control these contributions to the weak error.

However, the rate of convergence $\s_Q-\s>0$ of the third term in~\eqref{eq:weak} is twice the rate of convergence $(\s_Q-\s)/2$ of the third term in~\eqref{eq:strong}. These terms correspond with the discretization of the stochastic convolution, and this is where the difference between strong and weak orders of convergence appears. Observe that the second term always converges to $0$ faster than the third one: when $\s$ is an admissible parameter, $1-(\s_F+\s)/2\ge 1-\s\ge \s_Q-\s$. To have a similar control on the first term, one needs to assume that $\s_0$ is sufficiently large: $\s_0\ge 2\s_Q-\s$; this type of assumption is natural, since if $\s_0=\s$, then the convergence of $P_Nx_0$ to $x_0$ in $\HH^{\s}$ may be very slow.

As a consequence, if $\s_0\ge 2\s_Q-\s$, we have
$$\E\|u-u_N\|_{\infty,\s,T}\leq \frac{C_{\epsilon,\s}}{\lambda_{N+1}^{(\s_Q-\s-\epsilon)/2}} \;,\; \E[\Phi(u)]-\E[\Phi(u_N)]\leq \frac{C_{\epsilon,\s}(\Phi)}{\lambda_{N+1}^{(\s_Q-\s-\epsilon)}},$$
and thus the weak order of convergence is twice the strong order one.




\section{Error estimates for the damped-driven wave equation}\label{sect:wave}

In this section, we study the spectral Galerkin approximation of the following damped and stochastically driven wave equation, where we keep the notation and assumptions of Section~\ref{sect:Not_Ass}: 
\begin{equation}\label{eq:SPDE_wave}
\begin{cases}
du(t)=v(t)dt\\
dv(t)=-\gamma v(t)dt+Au(t)dt+F(u(t))dt+dW^{Q}(t).
\end{cases}
\end{equation}
We impose the
initial conditions $u(0)=u_0\in H$ and $v(0)=v_0\in \HH^{-1}$. 
The coefficient $\gamma\geq 0$ is a damping parameter.
The linear operators $A$ and $Q$ satisfy Assumptions~\ref{ass:A}~and~\ref{ass:B} respectively; however, we modify the assumption on $F$ as follows, changing
the range of allowable exponent:
\begin{hyp}\label{ass:F_wave}
For some $\s_F\in[0,1/2]$, the map $F:\HH^{\s_F}\rightarrow \HH^{-\s_F}$ is of class $\mathcal{C}^2$, with bounded first and second-order Fr\'echet derivatives.
\end{hyp}

Our aim is to show that the It\^o map technique used to prove strong and weak error estimates in spaces of trajectories $\mathcal{C}_{\s,T}$ in 
Sections~\ref{sect:strong_parabolic}~and~\ref{sect:weak_parabolic} 
can also be applied to this damped-driven wave equation to obtain results 
similar to Theorems~\ref{theo:strong} and~\ref{theo:weak}; our main result is Theorem~\ref{theo:wave}.
In order to be concise, we only sketch the main arguments of the analysis, since they are straightforward generalizations of the ones used in the previous sections in the parabolic case.
The following subsections tackle, in order, the notation employed,
the definition of solution and relation to an It\^o map, the Galerkin
approximation and the strong and weak error estimates.

\subsection{Notation}

Introduce the process $X(t)=\bigl( u(t) , v(t)\bigr)$, with values in $\hat{\HH}=\HH^{0}\times \HH^{-1}$. For $\s\in \R$, set $\hat{\HH}^{\s}=\HH^{\s}\times \HH^{\s-1}$, which is a Hilbert space with the scalar product defined by $\bigl[(u^1,v ^1),(u^2,v^2)\bigr]_{\s}=\langle u^1,u^2\rangle_{\s}+\langle v^1,v^2\rangle_{\s-1}$. The associated norm in $\hat{\HH}^\s$ is denoted by $|\cdot|_{\s}$.
Then the second-order SPDE~\eqref{eq:SPDE_wave} can be rewritten as the following first-order stochastic evolution equation in $\hat{\HH}$:
\begin{equation}\label{eq:SPDE_wave_X}
dX(t)=\mathcal{A}X(t)dt+\mathcal{F}(X(t))dt+d\mathcal{W}^{\mathcal{Q}}(t), \; X(0)=x_0=\bigl(u_0,v_0\bigr),
\end{equation}
where
\begin{equs}[2]
\mathcal{A}x&=(v,Au)\in \hat{\HH^0}  \; &\quad\text{for all}~ x&=\bigl(u,v\bigr)\in \hat{\HH}^{1},\\
\mathcal{F}(x)&=(0,F(u)-\gamma v) \in \hat{\HH}^{ \s_F  } \; &\quad\text{for all}~ x&=\bigl(u,v\bigr)\in \hat{\HH}^{\s_F},\\
\mathcal{Q}(x)&=(0,Qv) \; &\quad\text{for all}~ x&=\bigl(u,v\bigr)\in \hat{\HH}^{0},
\end{equs}
and the stochastic perturbation $\mathcal{W}^{\mathcal{Q}}$ is a $\mathcal{Q}$-Wiener process on $\hat{\HH}^0$. 
The unbounded linear operator $\mathcal{A}$ on $\mathcal{\HH}$ generates a group $\bigl(e^{t\mathcal{A}}\bigr)_{t\in \R}$, where, for all $t\in \R$, and $x=(u,v)\in \hat{\HH}$, 
$e^{t\mathcal{A}}x=\bigl(u^t,v^t\bigr)$ satisfies
\begin{gather*}
u^t=\sum_{k\in\N^*}\Bigl(\cos(t\sqrt{\lambda_k})\langle u,e_k\rangle+\sqrt{\lambda_{k}^{-1}}\sin(t\sqrt{\lambda_k})\langle v,e_k\rangle \Bigr)e_k,\\
v^t=\sum_{k\in\N^*}\Bigl(-\sqrt{\lambda_{k}}\sin(t\sqrt{\lambda_k})
\langle u,e_k\rangle+\langle \cos(t\sqrt{\lambda_k})\langle v,e_k\rangle \Bigr)e_k.
\end{gather*}


\subsection{Mild solutions and the It\^o map}


For any $\s\in\R$ and $T\in(0,+\infty)$, denote by $\hat{\mathcal{C}}_{\s,T}=\mathcal{C}\bigl([0,T],\hat{\HH}^{\s}\bigr)=\mathcal{C}_{\s,T}\times \mathcal{C}_{\s-1,T}$ the space of trajectories for $X$. The norm in $\hat{\mathcal{C}}_{\s,T}$ is still denoted by $\|\cdot\|_{\infty,\s,T}$.
The It\^o map $\hat{\Theta}$  associated with the wave equation \eqref{eq:SPDE_wave} is defined in Proposition~\ref{propo:Ito_wave_1} below (see Propositions~\ref{prop:eq_deter} and \ref{prop:eq_deter+w} in the parabolic case).

\begin{propo}\label{propo:Ito_wave_1}
Let Assumptions \ref{ass:A} and \ref{ass:F_wave} hold, and assume that $\s\in[\s_F,1/2]$.
Let $x_0=(u_0,v_0)\in \hat{\HH}^{\s_0}$ with $\s_0\geq \s$, $T\in(0,+\infty)$. Let $\hat{w}\in\hat{\mathcal{C}}_{\s,T}$. Then:
\begin{itemize}
\item
there exists a unique function $\hat{Y}^{\hat{w}}\in\hat{\mathcal{C}}_{\s,T}$ such that for any $t\geq 0$,
\begin{equation}\label{eq:Y^w_wave}
\hat{Y}^{\hat{w}}(t)=e^{t\mathcal{A}}x_0+\int_{0}^{t}e^{(t-s)\mathcal{A}}\mathcal{F}(\hat{Y}^{\hat{w}}(r)) dr+\hat{w}(t);
\end{equation}

\item
the mapping
$\hat{\Theta}:\hat{\mathcal{C}}_{\s,T} \rightarrow \hat{\mathcal{C}}_{\s,T}$
given by $\hat{w} \mapsto \hat{Y}^{\hat{w}}$
is of class $\mathcal{C}^2$,
with bounded Fr\'echet derivatives of first and second order.
\end{itemize}

\end{propo}

The following replaces Definition \ref{defi:admiss}
for the duration of this section; specifically it is used in
the proposition and theorem which follow.

\begin{defi}\label{defi:admiss2}
The parameter $\s \in \R^+$ is an admissible parameter if the following conditions are satisfied: $\s\in[\s_F,1/2]$ and $\s<\s_Q$, where $\s_Q$ is given in Definition~\ref{defi:s_Q}.
\end{defi}

\begin{propo}
Let $T\in(0,+\infty)$ and assume that $x_0=(u_0,v_0)\in \hat{\HH}^{\s_0}$, for some admissible $\s$.
Then \eqref{eq:SPDE_wave_X} admits a unique mild solution $X=(u,v)$, {\it i.e.} an $\hat{\HH}^{0}$-valued process such that, for any $0\leq t\leq T$,
\begin{equation}\label{eq:mild_wave}
X(t)=e^{t\mathcal{A}}x_0+\int_{0}^{t}e^{(t-s)\mathcal{A}}\mathcal{F}(X(r)) dr+\mathcal{W}^{\mathcal{A},\mathcal{Q}}(t),
\end{equation}
with the stochastic convolution $\mathcal{W}^{\mathcal{A},\mathcal{Q}}$ satisfying
\begin{equation}\label{eq:stoc_conv_wave}
\mathcal{W}^{\mathcal{A},\mathcal{Q}}(t)=\int_{0}^{t}e^{(t-r)\mathcal{A}}d\mathcal{W}^{\mathcal{Q}}(r).
\end{equation}
Moreover, $X$ admits a version in $\hat{\mathcal{C}}_{\s,T}$ 
such that we have the It\^o map representation
$$X=\hat{\Theta}\bigl(\mathcal{W}^{\mathcal{A},\mathcal{Q}}\bigr).$$
\end{propo}

\subsection{Galerkin discretization}

For any $\ell\in\N^*$, set
$\hat{e}_{2\ell-1}=(e_\ell,0)$ and $\hat{e}_{2\ell}=(0,\sqrt{\lambda_\ell}e_\ell)$. Then $\bigl(\hat{e}_{n}\bigr)_{n\in \N^*}$ is a complete orthonormal system of $\hat{\HH}^0$ such that $\mathcal{A}\hat{e}_{2\ell-1}=-\sqrt{\lambda_{\ell}}\hat{e}_{2\ell}$ and $\mathcal{A}\hat{e}_{2\ell}=\sqrt{\lambda_{\ell}}\hat{e}_{2\ell-1}$. For all $N\in \N^*$, define $\hat{\HH}_N=\text{span}\left\{ \hat{e}_k~;~k\in \left\{1,\ldots,2N\right\}\right\}$
and denote by $\mathcal{P}_N$ the associated orthogonal projector in $\hat{\HH}^0$ onto $\hat{\HH}_N$.

Define $X_N=\bigl(u_N,v_N\bigr)$ as the unique mild solution of
\begin{equation}
dX_N(t)=\mathcal{A}X_N(t)dt+\mathcal{P}_N\mathcal{F}(X(t))dt+\mathcal{P}_Nd\mathcal{W}^{\mathcal{Q}}(t), \; X_N(0)=\mathcal{P}_Nx_0.
\end{equation}
Equivalently,
\begin{equation}
\begin{cases}
du_N(t)=v_N(t)dt\\
dv_N(t)=-\gamma v_N(t)dt+Au_N(t)dt+P_NF(u(t))dt+P_NdW^{Q}(t)
\end{cases}
\end{equation}
with $u_N(0)=P_Nu_0$, $v_N(0)=P_Nv_0$.

Similarly to Theorem~\ref{theo:Ito_N} in the parabolic case, Proposition~\ref{propo:Ito_wave} below gives an expression of $X_N$ in terms of the It\^o map $\hat{\Theta}$.

\begin{propo}\label{propo:Ito_wave}
Set $\hat{R}_N(t)=(\mathcal{P}_N-I)e^{t\mathcal{A}}x_0+\int_{0}^{t}(\mathcal{P}_N-I)e^{(t-r)\mathcal{A}}\mathcal{F}(X_N(r))dr$. Then
$$X_N=\hat{\Theta}(\hat{R}_N+\mathcal{P}_N\mathcal{W}^{\mathcal{A},\mathcal{Q}}).$$
Moreover, set $\overline{\s}=\min(1/2,\s_0,\s_Q)$. Then for any $\epsilon\in(0,\overline{\s})$, the following moment estimate is satisfied: there exists $C_{\overline{\s}-\epsilon,T}\in(0,+\infty)$ such that
$$\sup_{N\in\N^*}\E \big\|X_N\|_{\infty,\overline{\s}-\epsilon,T}\leq C_{\s,T}.$$
\end{propo}

\subsection{Strong and weak error estimates}\label{sect:orders_wave}

We are now in position to state strong and weak error estimates for the convergence of $X_N=(u_N,v_N)$ to $X=(u,v)$ in a space of trajectories $\hat{\mathcal{C}}_{\s,T}$, with appropriate orders of convergence. 
Note that for the weak convergence result,
one can choose test functions $\Phi$ 
depending only on the $u$-component of $X=(u,v)$.

\begin{theo}\label{theo:wave}
Let $x_0=(u_0,v_0)\in \hat{\HH}^{\s_0}$, $T\in(0,+\infty)$ and $\s$ be admissible, with $\s_0\geq \s$. Define $\overline{\s}=\min(1/2,\s_0,\s_Q)$.
Let $\Phi:\hat{\mathcal{C}}_{\s,T}\rightarrow \R$ be an \emph{admissible test function}, \emph{i.e.} it is bounded and of class $\mathcal{C}^2$, with bounded Fr\'echet derivatives of first and second order.
Then for any $\epsilon\in(0,\overline{\s}\wedge(\s_Q-\s))$, there exists $C_{\epsilon,\s},C_{\epsilon,\s}(\Phi)\in(0,+\infty)$, such that for any $N\in\N^*$:
\begin{itemize}
\item the following strong error estimate is satisfied
\begin{equation}\label{eq:strong_wave}
\E\|X-X_N\|_{\infty,\s,T}\leq C_{\epsilon,\s}\Bigl(\frac{1}{\lambda_{N+1}^{(\s_0-\s)/2}}|x_0|_{\s_0}+\frac{1}{\lambda_{N+1}^{(\overline{\s}-\s-\epsilon)/2}}+\frac{1}{\lambda_{N+1}^{(\s_Q-\s-\epsilon)/2}}\Bigr).
\end{equation}
\item the following weak error estimate is satisfied
\begin{equation}\label{eq:weak_wave}
\big|\E[\Phi(X)]-\E[\Phi(X_N)]\big|\leq C_{\epsilon,\s}(\Phi)\Bigl(\frac{1}{\lambda_{N+1}^{(\s_0-\s)/2}}|x_0|_{\s_0}+\frac{1}{\lambda_{N+1}^{(\overline{\s}-\s-\epsilon)/2}}+\frac{1}{\lambda_{N+1}^{\s_Q-\s-\epsilon}}\Bigr).
\end{equation}

\end{itemize}
\end{theo}

The proof of the above is very similar to the parabolic case and we omit it.
The key argument to obtain the weak error estimate \eqref{eq:weak_wave} using Proposition~\ref{propo:Ito_wave} is the independence of the $\hat{\mathcal{C}}_{\s,T}$-valued random variables $\mathcal{P}_N\mathcal{W}^{\mathcal{A},\mathcal{Q}}$ and $(I-\mathcal{P}_N)\mathcal{W}^{\mathcal{A},\mathcal{Q}}$: indeed the former depends only on $\{\beta_i\,:\, i \le N\}$, while the latter
only depends on $\{\beta_i\,:\, i > N\}$.

\begin{rem}
When $\gamma=0$, one may replace the definition of $\overline{\s}$ with $\overline{\s}=\min(1-\s_F,\s_0,\s_Q)$.
Note that in the general case $\gamma>0$, $1-\s_F\geq 1/2\geq \overline{\s}$: as a consequence, the orders of convergence in \eqref{eq:strong_wave} and \eqref{eq:weak_wave} do not depend on $\s_F$.
\end{rem}

\section{Parabolic Equation with non-commuting noise}\label{sect:non_comm}



In this section, we return to the parabolic SPDE~\eqref{eq:SPDE}, studied in Sections~\ref{sect:Ito_parabolic} and~\ref{sect:Galerkin_parabolic}, and we remove the assumption that $A$ and $Q$ commute (see Assumption~\ref{ass:B}). Instead, from now on we consider Assumption~\ref{ass:B_gen} below.
Note that the strong convergence analysis (Theorem~\ref{theo:strong}) is not modified. However, the assumption that $A$ and $Q$ commute has been crucial to prove Theorem~\ref{theo:weak}; now, under Assumption~\ref{ass:B_gen}, the processes $P_NW^{A,Q}$ and $P_{N}^{\perp}W^{A,Q}$ are not necessarily independent.
In order to generalize Theorem~\ref{theo:weak}, we introduce two auxiliary processes:
\begin{itemize}
\item $W^{A,Q,(N)}$ defined by~\eqref{eq:truncated} is the stochastic convolution with truncated noise $P_NW$ instead of $W$.
\item $\rho_N=W^{A,Q,(N)}-P_NW^{A,Q}$, which involves a commutator. 
\end{itemize}

We first state and prove a general result, Theorem~\ref{theo:weak_gen}, where we repeat the arguments of the proof of Theorem~\ref{theo:weak}; an extra error term appears when effectively $Q$ and $A$ do not commute. In general, we are not able to identify the order of convergence. We thus consider a specific case, where explicit computations allow us to prove that the weak order of convergence is twice the strong one, see Theorem~\ref{theo:weak_special}.


\subsection{Assumptions}


The assumption that $A$ and $Q$ do not necessarily commute is expressed in Assumption~\ref{ass:B_gen} below, thanks to the introduction of a new complete orthonormal system $(f_k)_{k\in \N^*}$ of $H$.

\begin{hyp}\label{ass:B_gen}
There exists a complete orthonormal system $(f_k)_{\in\N^*}$, and a bounded sequence $(q_k)_{k\in \N^*}$ of nonnegative real numbers, such that the bounded linear operator $Q\in\mathcal{L}(H)$ satisfies for any $k\in \N^*$
$$Qf_k=q_kf_k.$$
Define also the self-adjoint operator $B\in\mathcal{L}(H)$ as follows: for all $x\in H$
$$Bx=\sum_{k=1}^{+\infty}\sqrt{q_k}\langle x,f_k \rangle f_k.$$
\end{hyp}

Note that the $Q$-cylindrical Wiener process $W^Q$ is such that
\begin{equation}\label{eq:expansions_bis}
\begin{gathered}
W^Q(t)=\sum_{k\in\N^*}\sqrt{q_k}\tilde{\beta}_{k}(t)f_k=BW(t),\\
W(t)=\sum_{k\in\N^*}\beta_{k}(t)e_k=\sum_{k\in\N^*}\tilde{\beta}_{k}(t)f_k,
\end{gathered}
\end{equation}
where $\bigl(\tilde{\beta}_k\bigr)_{k\in\N^*}$ is a sequence of independent, standard, real-valued Wiener processes.
The spectral Galerkin discretization of the unique mild solution $X$ of ~\eqref{eq:SPDE} is still defined by projection onto $H_N$, see~\eqref{eq:SPDE_projN}.

Let $\Phi$ (see Definition~\ref{defi:test_funct}) be an admissible test function. The aim is to study the weak error $e_N(\Phi,\s)$ defined by~\eqref{eq:e_N_Phi}. Following the proof of Theorem~\ref{theo:weak}, using the Lipschitz continuity of $\Phi\circ \Theta$,
$$|e_N(\Phi,\s)|\leq C\E\|R_N\|_{\infty,\s,T}+\tilde{e}_N(\Phi,\s),$$
with
\begin{equation}\label{eq:tilde_e_N_Phi}
\tilde{e}_N(\Phi,\s)=\E[\Phi\circ\Theta(W^{A,Q})]-\E[\Phi\circ\Theta(P_NW^{A,Q})].
\end{equation}
The error term $\E\|R_N\|_{\infty,\s,T}$ is controlled in Theorem~\ref{theo:strong}. Thus in the sequel we focus on controlling the auxiliary weak error~\eqref{eq:tilde_e_N_Phi}.



\subsection{Truncation of the noise}


Introduce the Wiener process $W^{(N)}$ and the associated stochastic convolution $W^{A,Q,(N)}$ with truncation at level $N$:
\begin{equation}\label{eq:truncated}
\begin{gathered}
W^{(N)}(t)=P_NW(t)=\sum_{n=1}^{N}\beta_n(t)e_n,\\
W^{A,Q,(N)}(t)=\int_{0}^{t}e^{(t-r)A}BdW^{(N)}(r).
\end{gathered}
\end{equation}
The process $W^{A,Q,(N)}$ is equal to the stochastic convolution $W^{A,Q_N}$, with the covariance operator $Q_N=BP_NB$.

The key observation is the independence of the $\mathcal{C}_{\s,T}$-valued random variables $W^{A,Q,(N)}$ and $W^{A,Q}-W^{A,Q,(N)}$, since the former depends only on $\{\beta_i\,:\, i \le N\}$, while the latter only depends on $\{\beta_i\,:\, i > N\}$.
Applying the strategy of the proof of Theorem~\ref{theo:weak}, we obtain, with $\Psi=\Phi\circ \Theta$,
\begin{equs}
\Big|\E[\Phi&\circ\Theta(W^{A,Q})]-\E[\Phi\circ\Theta(W^{A,Q,(N)})]\Big|\\&=\Big|\E\Bigl(\Psi(W^{A,Q})-\Psi(W^{A,Q,(N)})-D\Psi(W^{A,Q,(N)}).[W^{A,Q}-W^{A,Q,(N)}]\Bigr)\Big|\\
&\leq C(\Psi)\E\|W^{A,Q}-W^{A,Q,(N)}\|_{\infty,\s,T}^2.\label{eq:weak_error_trunc}
\end{equs}
The inequality above is interesting, but is not sufficient for our purpose. Indeed, the stochastic convolution $W^{A,Q,(N)}$ does not {\it a priori} take values in $H_N$.

\subsection{The general weak error analysis}

We now decompose the weak error \eqref{eq:tilde_e_N_Phi} as follows:
\begin{equation}\label{eq:decomp_e_tilde}
\begin{aligned}
\tilde{e}_N(\Phi,\s)&=\E[\Phi\circ\Theta(W^{A,Q})]-\E[\Phi\circ\Theta(P_NW^{A,Q})]\\
&=\E[\Phi\circ\Theta(W^{A,Q})]-\E[\Phi\circ\Theta(W^{A,Q,(N)})\\
&+\E[\Phi\circ\Theta(W^{A,Q,(N)})-\E[\Phi\circ\Theta(P_NW^{A,Q})].
\end{aligned}
\end{equation}
Observe that for any $t\in[0,T]$
$$\rho_N(t):=W^{A,Q,(N)}(t)-P_{N}W^{A,Q}(t)=\int_{0}^{t}e^{(t-r)A}[B,P_N]dW(r),$$
with the \emph{commutator} $[B,P_N]=BP_N-P_NB$.

The first error term in \eqref{eq:decomp_e_tilde} is controlled thanks to \eqref{eq:weak_error_trunc}. Moreover, using the Lipschitz continuity of $\Phi\circ\Theta$, there exists $C_{\s}(\Phi)\in(0,+\infty)$ such that for any $N\in\N^*$
\begin{equation}\label{eq:error_commu}
\E[\Phi\circ\Theta(W^{A,Q,(N)})-\E[\Phi\circ\Theta(P_NW^{A,Q})]\leq C_{\s}(\Phi)\E\|\rho_N\|_{\s,T,\infty}.
\end{equation}
Note that using a second-order Taylor expansion would not lead to an improved order of convergence for this term.

As a consequence, \eqref{eq:weak_error_trunc} and \eqref{eq:error_commu} give for any $N\in\N^*$
\begin{equation}\label{eq:weak_error_decompo}
\begin{aligned}
\big|\tilde{e}_N(\Phi,\s)\big|&\leq C\Bigl(\E\|\rho_N\|_{\s,T,\infty}+\E\|W^{A,Q}-W^{A,Q,(N)}\|_{\s,T,\infty}^{2}\Bigr)\\
&\leq C_{\s}(\Phi)\Bigl(\E\|\rho_N\|_{\s,T,\infty}+2\E\|\rho_N\|_{\s,T,\infty}^2+2\E\|W^{A,Q}-P_NW^{A,Q}\|_{\s,T,\infty}^{2}\Bigr).
\end{aligned}
\end{equation}

It is not difficult to prove that $\E\|\rho_N\|_{\s,T,\infty}^2\underset{N\to +\infty}\to 0$. However, the identification of the order of convergence in the general case is not easy. In Section~\ref{sect:example_non_comm} below, we perform this task in a specific example.


Thus we have proved the following general result.

\begin{theo}\label{theo:weak_gen}
Let $T\in(0,+\infty)$. Let Assumptions \ref{ass:A}, \ref{ass:F} and \ref{ass:B_gen} hold. Assume that the initial condition satisfies $u_0\in\HH^{\s_0}$, and that $\s$ is an admissible parameter (see Definition \ref{defi:admiss}), with $\s_0\geq \s$.
Let $\Phi:\mathcal{C}_{\s,T}\rightarrow \R$  be an admissible test function. Then for any $\epsilon\in(0,\s_Q-\s)$, there exists a constant $C_{\epsilon,\s}(\Phi)\in(0,+\infty)$, such that for any $N\in\N^*$
\begin{equation}\label{eq:weak_gen}
\begin{aligned}
\big|e_N(\Phi,\s)\big|&\leq C_{\epsilon,\s}(\Phi)\Bigl(\frac{1}{\lambda_{N+1}^{(\s_0-\s)/2}}|u_0|_{\s_0}+\frac{1}{\lambda_{N+1}^{1-(\s_F+\s+\epsilon)/2}}+\frac{1}{\lambda_{N+1}^{\s_Q-\s-\epsilon}}\Bigr)\\
&~\hspace{1.2in}+C_{\epsilon,\s}(\Phi)\bigl(\E\|\rho_N\|_{\s,T,\infty}^2\bigr)^{1/2}.
\end{aligned}
\end{equation}
\end{theo}

The first part of the right-hand side of~\eqref{eq:weak_gen} exactly corresponds with the right-hand side of~\eqref{eq:weak} (since $\rho_N=0$ in this case): the non-commutation of $A$ and $Q$ makes the second part of the right-hand side of~\eqref{eq:weak_gen} appear. If this fourth error term converges faster to $0$ than the third one, one thus recovers the same rates of convergence as in the commuting case.





\subsection{Identification of the weak order for a specific non-commuting example}\label{sect:example_non_comm}

We consider the following SPDE with periodic boundary conditions on $(0,1)$ driven by a $Q$-Wiener process $W^Q$:
\begin{equation}\label{eq:SPDE_special}
du(t)=Au(t)dt+F(u(t))dt+dW^Q(t) \; , \; X(0)=0.
\end{equation}


We slightly modify the framework exposed in Section~\ref{sect:Not_Ass}, in order to simplify computations below: we now consider the Hilbert space of square integrable complex-valued functions. The new framework is given by Assumption~\ref{ass:special}.

\begin{hyp}\label{ass:special}
\begin{itemize}
\item The state space $H=L^{2}(\mathbb{T})$, with $\mathbb{T}=\R/\Z$, is the space of $\mathbb{C}$-valued square integrable functions, with sesquilinear form given by
$$\langle f,g \rangle=\int_{0}^{1}f(\xi)\overline{g(\xi)}d\xi.$$
\item The Fourier basis is denoted as follows: for any $n\in\Z$ ($\Z$ denotes the set of integers) and any $\xi\in\R$
$$e_n(\xi)=\exp(2i\pi n\xi).$$
\item For any $N\in\N^*$, $H_N=\text{span}\left\{e_{-N},\ldots,e_0,\ldots,e_N\right\}$, and $P_N$ is the orthogonal projector on $H_N$.
\item The linear operator $A$ satisfies for any $n\in\Z$
$$Ae_n=-\lambda_ne_n$$
with $\lambda_n=4n^2\pi^2+1$.
\item Let $b:\mathbb{T}\rightarrow\R$ be a function of class $\mathcal{C}^{\rr}$, for some $\rr\geq 2$, such that $b>0$ everywhere. Then $B:H\rightarrow H$ is the multiplication operator $x\mapsto Bx=b.x$. $B$ is a self-adjoint bounded operator; moreover it is invertible, with $B^{-1}:x\mapsto b^{-1}.x$.
\item $b$ is not constant: as a consequence $A$ and $B$ do not commute.
\item The Fourier coefficients of $b$ are defined as follows: for $n\in\Z$,  $b_n=\langle b,e_n\rangle=\int_{0}^{1}b(\xi)\exp(-2i\pi n\xi)d\xi$.
\item The covariance operator $Q$ is defined as $Q=B^2$.
\end{itemize}
\end{hyp}
Note that there exists $C_{\rr}\in(0,+\infty)$ such that for any $n\in\Z$
$$|b_n|\leq C_{\rr}|n|^{-\rr}\leq C_{\rr}\lambda_{n}^{-\rr/2}.$$
Due to Assumption~\ref{ass:special}, $B$ is an invertible operator: as a consequence the process $X$ defined by \eqref{eq:SPDE_special} possesses the same regularity as in the case when $B=I$: this gives $\s_Q=1/2$. Indeed, for any $\s\in\R$,
$$\Tr(A^{\s-1}Q)=\Tr(BA^{\s-1}B)=|b|_{H}^{2}\Tr(A^{\s-1}),$$
where $|b|_{H}^{2}=\int_{0}^{1}|b(\xi)|^2d\xi$.

According to Theorem~\ref{theo:weak_gen}, we need to control $\E\|\rho_N\|_{\infty,\s,T}^2$, where
$$\rho_N(t)=\int_{0}^{t}e^{(t-r)A}[B,P_N]dW(r).$$
The main result required to prove a weak convergence theorem
for the specific non-commuting example is the following lemma whose
proof is deferred to a final subsection.

\begin{lemme}\label{lem:weak_special_comm}
Under Assumption \ref{ass:special}, for any $\epsilon\in(0,1/2-\s)$ there exists a constant $C_{\epsilon,\s,T}\in(0,+\infty)$ such that for any $N\in\N^{*}$
\begin{equation}\label{eq:weak_special_comm}
\E\|\rho_N\|_{\infty,\s,T}^2\leq \frac{C_{\epsilon,\s,T}}{\lambda_{N}^{1-\s-\epsilon}}.
\end{equation}
\end{lemme}

Since $\s_Q=1/2$ and $\s\ge 0$, we have $1/2-\s/2-\epsilon/2\ge \s_Q-\s-\epsilon$: this means that in the example treated in this section, the fourth error term in~\eqref{eq:weak_gen} is bounded from above by the third error term (up to a multiplicative constant). 
Thanks to Theorem~\ref{theo:weak_gen}, we thus immediately obtain the following generalization of Theorem~\ref{theo:weak} in a non-commuting example, with the same rates of convergence.
\begin{theo}\label{theo:weak_special}
Under Assumptions \ref{ass:special} and \ref{ass:F},let $u_0\in\HH^{\s_0}$, and let $\s$ be an admissible parameter with $\s_0\geq \s$. Let $\Phi:\mathcal{C}_{\s,T}\rightarrow \R$  be an admissible test function.
Then for any $\epsilon\in(0,\s_Q-\s)$, there exists a constant $C_{\epsilon,\s}(\Phi)\in(0,+\infty)$, not depending on $N\in\N^*$, such that
\begin{equation}\label{eq:weak_special}
\big|e_N(\Phi,\s)\big|\leq C_{\epsilon,\s}(\Phi)\Bigl(\frac{1}{\lambda_{N+1}^{(\s_0-\s)/2}}|u_0|_{\s_0}+\frac{1}{\lambda_{N+1}^{1-(\s_F+\s+\epsilon)/2}}+\frac{1}{\lambda_{N+1}^{1/2-\s-\epsilon}}\Bigr).
\end{equation}
\end{theo}

%
%
%
%

\begin{proof}[Proof of Lemma~\ref{lem:weak_special_comm}]
We use the factorization formula
$$\rho_N=\Gamma\bigl(\mathcal{Z}_N\bigr),$$
where $\Gamma$ is defined by \eqref{eq:def_Gamma} and $\mathcal{Z}_N$ is the Gaussian process such that for all $t\in[0,T]$
$$\mathcal{Z}_N(t)=\int_{0}^{t}(t-r)^{-\epsilon/2}e^{(t-r)A}[B,P_N]dW(r).$$
To obtain a control of $\E\|\rho_N\|_{0,T,\infty}^{2}$, it thus suffices to bound
$\sup_{0\leq t\leq T}\E|\mathcal{Z}_{N}(t)|_{\s}^{2}$.
For any $t\in[0,T]$, by It\^o's formula from \cite{DPZ} which
applies under the conditions of this section,
\begin{align*}
\E|\mathcal{Z}_N(t)|_{\s}^{2}&=\int_{0}^{t}\frac{1}{(t-r)^{\epsilon}}\Tr\Bigl((-A)^{\s/2}e^{(t-r)A}[B,P_N][B,P_N]^{T}e^{(t-r)A}(-A)^{\s/2}\Bigr)dr\\
&=\int_{0}^{t}\frac{1}{(t-r)^{\epsilon}}\sum_{k\in\Z}\big|(-A)^{\s/2}e^{(t-r)A}[B,P_N]e_k\big|^2 dr\\
&=\int_{0}^{t}\frac{1}{(t-r)^{\epsilon}}\sum_{k,\ell\in\Z}\big|\langle(-A)^{\s/2}e^{(t-r)A}[B,P_N]e_k,e_{\ell}\rangle\big|^2 dr\\
&= \sum_{k,\ell\in\Z}\int_{0}^{t}\big|\langle[B,P_N]e_k,e_{\ell}\rangle\big|^2\frac{\lambda_{\ell}^{\s}e^{-2\lambda_{\ell}(t-r)}}{(t-r)^{\epsilon}} dr.
\end{align*}
Straightforward computations give
\begin{equs}
be_k&=\sum_{\ell\in\Z}\langle be_k,e_{\ell}\rangle e_{\ell}=\sum_{\ell\in\Z}b_{\ell-k}e_{\ell}\\
\big|\langle[B,P_N]e_k,e_{\ell}\rangle\big|&=|b_{\ell-k}|\Bigl(\mathds{1}_{|k|>N}\mathds{1}_{|\ell |\leq N}+\mathds{1}_{|k|\leq N}\mathds{1}_{|\ell |> N}\Bigl).
\end{equs}
As a consequence, the following equality is satisfied: for any $N\in\N^*$
\begin{equation}\label{eq:expand_error_comm}
\begin{aligned}
\E|\mathcal{Z}_N(t)|_{\s}^{2}&=\sum_{|k|\leq N}\sum_{|\ell|>N}|b_{\ell-k}|^2\int_{0}^{t}\frac{\lambda_{\ell}^{\s}e^{-2\lambda_{\ell}(t-r)}}{(t-r)^{\epsilon}}dr\\
&~+\sum_{|k|>N}\sum_{|\ell|\leq N}|b_{\ell-k}|^2\int_{0}^{t}\frac{\lambda_{\ell}^{\s}e^{-2\lambda_{\ell}(t-r)}}{(t-r)^{\epsilon}}dr.
\end{aligned}
\end{equation}

We now prove bounds on the terms of the right-hand side of~\eqref{eq:expand_error_comm}.
Note that there exists $C_{\epsilon,\s,T}\in(0,+\infty)$ such that for any $\ell\in\Z$
$$\int_{0}^{t}\frac{\lambda_{\ell}^{\s}e^{-2\lambda_{\ell}(t-r)}}{(t-r)^{\epsilon}}dr\leq \frac{C_{\epsilon,\s,T}}{\lambda_{\ell}^{1-\s-\epsilon}}\leq \frac{C_{\epsilon,\s,T}}{(|\ell|+1)^{2(1-\s-\epsilon)}}.$$
The first term in \eqref{eq:expand_error_comm} is bounded as follows: for any $N\in\N^*$
\begin{equs}
\sum_{|k|\leq N}\sum_{|\ell|>N}&|b_{\ell-k}|^2\int_{0}^{t}\frac{\lambda_{\ell}^{\s}e^{-2\lambda_{\ell}(t-r)}}{(t-r)^{\epsilon}}dr \leq C\sum_{|k|\leq N}\sum_{| \ell |> N}\frac{1}{|\ell-k|^{2\rr}}\frac{1}{(|\ell|+1)^{2-2\s-2\epsilon}}\\
&\leq \sum_{\ell=N+1}^{+\infty}\frac{C}{\ell^{2-2\s-2\epsilon}}\sum_{k\leq N}\frac{1}{(\ell-k)^{2\rr}}
\leq \frac{C}{N^{2-2\s-2\epsilon}}\sum_{\ell=N+1}^{+\infty}\frac{1}{(\ell-N)^{2\rr-1}}
\leq \frac{C}{\lambda_{N+1}^{1-\s-\epsilon}},
\end{equs}
where the last series converges since $2\rr-1>1$.


The second term in \eqref{eq:expand_error_comm} is bounded similarly: for any $N\in\N^*$,
\begin{align*}
\sum_{|k|>N}\sum_{|\ell|\leq N}|b_{\ell-k}|^2\int_{0}^{t}\frac{\lambda_{\ell}^{\s}e^{-2\lambda_{\ell}(t-r)}}{(t-r)^{\epsilon}}dr&\leq C\sum_{|k|>N}\sum_{|\ell|\leq N}\frac{1}{|\ell-k|^{2\rr}}\frac{1}{(|\ell|+1)^{2-2\s-2\epsilon}}\\
&\leq C\sum_{k=N+1}^{+\infty}\sum_{|\ell|\leq N}\frac{1}{(k-\ell)^{2\rr}}\frac{1}{(|\ell|+1)^{2-2\s-2\epsilon}}.
\end{align*}
On the one hand,
$$\sum_{k=N+1}^{+\infty}\sum_{-N\leq \ell \leq N/2}\frac{1}{(k-\ell)^{2\rr}}\frac{1}{(|\ell|+1)^{2-2\s-2\epsilon}}\leq \sum_{k=N+1}^{+\infty}\frac{CN}{k^{2\rr}}\leq \frac{C}{N^{2\rr-2}}\leq \frac{C}{\lambda_{N+1}^{\rr-1}}.$$
On the other hand,
\begin{equ}
\sum_{k=N+1}^{+\infty}\sum_{\ell=N/2}^{N}\frac{1}{(k-\ell)^{2\rr}}\frac{1}{(\ell+1)^{2-2\s-2\epsilon}}\leq \frac{C}{N^{2-2\s-2\epsilon}}\sum_{j=1}^{+\infty}\frac{j}{j^{2\rr}}\leq \frac{C}{\lambda_{N+1}^{1-\s-\epsilon}}\;,
\end{equ}
since $2\rr-1> 1$.
Thus for any $t\in[0,T]$
$$\E|\mathcal{Z}_N(t)|_{\s}^{2}\leq C\frac{C}{\lambda_{N+1}^{1-\s-\epsilon}}.$$
This concludes the proof of Lemma~\ref{lem:weak_special_comm}.
\end{proof}

\section*{Acknowledgments}

C.-E.B. would like to thank the Warwick Mathematics Institute for their hospitality during his visit in 2013 and 
R\'egion Bretagne for funding this visit. The
work of A.M.S. was supported by EPSRC, ERC and ONR. M.H. was supported by the ERC and the Philip Leverhulme trust.

%
%

\end{document}